\documentclass[11pt]{amsart}

\usepackage{amsmath, amssymb, amsthm, amsfonts, amsxtra, color, extarrows, hyperref, mathrsfs, mathtools, tikz, tikz-cd}
\usepackage[foot]{amsaddr}

\input xy
\xyoption{all}

\theoremstyle{plain}
\newtheorem{theorem}{Theorem}
\newtheorem{corollary}{Corollary}
\newtheorem{lemma}{Lemma}

\newtheorem{proposition}{Proposition}
\theoremstyle{definition}
\newtheorem{conjecture}{Conjecture}

\newtheorem{example}{Example}

\newtheorem{remark}{Remark}
\newtheorem{question}{Question}

\numberwithin{equation}{section}
\numberwithin{theorem}{section}
\numberwithin{corollary}{section}
\numberwithin{conjecture}{section}
\numberwithin{lemma}{section}
\numberwithin{proposition}{section}
\numberwithin{remark}{section}
\numberwithin{example}{section}
\numberwithin{question}{section}

\hypersetup{}

\def\ZZ{\mathbb{Z}}

\newcommand{\abs}[1]{\lvert{#1}\rvert}

\renewcommand{\a}{\alpha}

\title{Minimum Coprime Graph Labelings}
\author{Catherine Lee}
\address{Department of Mathematics, Yale University, New Haven, CT 06511}
\email{catherine.lee@yale.edu}

\begin{document}
\maketitle

\begin{abstract}
A coprime labeling of a graph $G$ is a labeling of the vertices of $G$ with distinct integers from $1$ to $k$ such that adjacent vertices have coprime labels. The minimum coprime number of $G$ is the least $k$ for which such a labeling exists. In this paper, we determine the minimum coprime number for a few well-studied classes of graphs, including the coronas of complete graphs with empty graphs and the joins of two paths. In particular, we resolve a conjecture of Seoud, El Sonbaty, and Mahran and two conjectures of Asplund and Fox. We also provide an asymptotic for the minimum coprime number of the Erd\H{o}s--R\'enyi random graph.
\end{abstract}

\section{Introduction} \label{intro}
\subsection{Background}
Let $G$ be a simple graph with $n = \abs{V(G)}$ vertices. A \emph{coprime labeling} of $G$ is an injection $f \colon V(G) \to \{1, 2, \dots, k\}$, for some integer $k \geq n$, such that if $(u,v) \in E(G)$ then $\gcd(f(u),f(v)) = 1$. The \emph{minimum coprime number} $\mathfrak{pr}(G)$ is the least $k$ for which such a labeling exists; a coprime labeling of $G$ using only integers up to $\mathfrak{pr}(G)$ is called a \emph{minimum coprime labeling} of $G$. If $\mathfrak{pr}(G) = n$, then a minimum coprime labeling of $G$ is called a \emph{prime labeling} and $G$ is a \emph{prime graph}.

The notion of prime labeling originated with Entringer and was introduced in a paper by Tout, Dabboucy, and Howalla \cite{tout}. It is conceptually related to the \emph{coprime graph of integers}, the graph with vertex set $\ZZ$ that contains the edge $(m,n)$ if and only if $\gcd(m,n)=1$. The induced subgraph $G(A)$ with vertex set $A \subset \{1, \dots, N\}$ is called the \emph{coprime graph of $A$} and was first studied by Erd\H{o}s \cite{erdos}, who posed the famous problem of finding the largest set $A \subset \{1, \dots, N\}$ such that $K_k \not\subset G(A)$. Newman's coprime mapping conjecture, which was proven by Pomerance and Selfridge \cite{pomerance}, involves the existence of perfect matchings in $G(A)$. Various other properties of the coprime graph of integers have been studied by Ahlswede and Khachatrian \cite{ahlswede1, ahlswede2, ahlswede3}, Erd\H{o}s \cite{erdos2}, Erd\H{o}s and A. S\'ark\"ozy \cite{erdos3}, Erd\H{o}s and G. N. S\'ark\"ozy \cite{erdos6}, Erd\H{o}s, A. S\'ark\"ozy, and Sz\'emeredi \cite{erdos4, erdos5}, and G. N. S\'ark\"ozy \cite{sarkozy}. Further discussion of the coprime graph of integers may be found in \cite{handbook}.

The problem of finding a minimum coprime labeling of a graph $H$ is equivalent to showing that $H$ is a subgraph of the coprime graph $G(\{1, 2, \dots, \mathfrak{pr}(H)\})$.
Much work has been done to prove that various classes of graphs are prime; we refer the reader to \cite{gallian} for a full catalog of results. In particular, it is known that all paths, cycles, helms, fans, flowers, books, and wheels of even order are prime \cite{gallian, seoud, seoud2}.

The primality of trees has been especially well studied. Entringer and Tout conjectured around 1980 that every tree is prime. While this conjecture remains open, it is now known that many classes of trees are prime, including paths, stars, caterpillars, spiders, and complete binary trees \cite{gallian}. Fu and Huang \cite{fu} proved in 1994 that trees with 15 or fewer vertices are prime, and Pikhurko \cite{pikhurko} extended this result to trees with 50 or fewer vertices in 2007. Salmasian \cite{salmasian} showed that any tree $T$ with $n \geq 50$ vertices satisfies $\mathfrak{pr}(T) \leq 4n$. Pikhurko \cite{pikhurko} improved upon this by showing that the Entringer--Tout conjecture holds asymptotically, i.e., for any $c>0$, there is an $N$ such that for any tree $T$ of order $n>N$, $\mathfrak{pr}(T) < (1+c)n$. In 2011, Haxell, Pikhurko, and Taraz \cite{haxell} proved the Entringer--Tout conjecture for trees of sufficiently large order.

In this paper, we focus on the minimum coprime numbers of a few well-studied classes of graphs. The first class we consider is formed by taking the coronas of complete graphs and empty graphs.
The \emph{corona} of a graph $G$ with a graph $H$, denoted $G \odot H$, is obtained by combining one copy of $G$ with $\abs{V(G)}$ copies of $H$ by attaching the $i^\text{th}$ vertex in $G$ to every vertex in the $i^\text{th}$ copy of $H$. Tout, Dabboucy, and Howalla \cite{tout} showed that the \emph{crown graphs} $C_n \odot \overline{K}_m$ are prime for all positive integers $n$ and $m$. The graphs $K_n \odot \overline{K}_m$, which are spanning supergraphs of the crowns $C_n \odot \overline{K}_m$, have also been studied in this context. Youssef and Elsakhawi \cite{youssef} showed that $K_n \odot K_1$ is prime for $n \leq 7$, and $K_n \odot \overline{K}_2$ is prime for $n \leq 16$.
Seoud, El Sonbaty, and Mahran \cite{seoud3} then observed that $K_n \odot \overline{K}_m$ is not prime if $n > \pi(n(m+1)+1)$.\footnote{Throughout this paper, we denote $\pi(x)$ to be the number of primes less than or equal to $x$, and $p_n$ to be the $n^\text{th}$ prime number.} They also conjectured the converse,
\begin{conjecture}[\cite{seoud3}, Conjecture 3.9] \label{conj1}
The graph $K_n \odot \overline{K}_m$ is prime if $n \leq \pi(n(m+1))+1$.
\end{conjecture}
They computed all values of $n$ satisfying this condition for all $m \leq 20$.
Most recently, Asplund and Fox \cite{asplund} computed the minimum coprime numbers of $K_n \odot K_1$ and $K_n \odot \overline{K}_2$, showing that $\mathfrak{pr}(K_n \odot K_1) = p_{n-1}$ if $n>7$ and $\mathfrak{pr}(K_n \odot \overline{K}_2) = p_{n-1}$ if $n>16$. They conjecture that their results extend whenever $n$ is sufficiently large relative to $m$.
\begin{conjecture}[\cite{asplund}, Conjecture 1] \label{conj2}
For all positive integers $m$, there exists an $M>m$ such that for all $n>M$, $\mathfrak{pr}(K_n \odot \overline{K}_m) = p_{n-1}$.
\end{conjecture}

We prove the following statement, which resolves Conjectures \ref{conj1} and \ref{conj2} affirmatively.
\begin{theorem} \label{corona-thm}
For all positive integers $m$ and $n$,
\begin{equation*}
\mathfrak{pr}(K_n \odot \overline{K}_m) = \max(mn+n, p_{n-1}).
\end{equation*}
\end{theorem}

%{\color{red} Something about this being stronger than Conj 2?}
%In particular, Theorem \ref{corona-thm} resolves Conjecture \ref{conj1} affirmatively, as we have
%\begin{corollary}
%$K_n \odot \overline{K}_m$ is prime if and only if $mn+n > p_{n-1}$.
%\end{corollary}

The next class of graphs we consider is constructed via the join operation. The \emph{join} of two disjoint graphs $G$ and $H$, denoted $G+H$, consists of the graph union $G \cup H$ with edges added to connect each vertex in $G$ to each vertex in $H$. We focus on the joins of two paths, which are well-studied. It was shown in \cite{seoud, seoud3} that $P_n + K_1 = P_n + P_1$ is prime, $P_n + \overline{K}_2$ is prime if and only if $n \geq 3$ is odd, and $P_n + \overline{K}_m$ is not prime for $m \geq 3$. Asplund and Fox \cite{asplund} computed the minimum coprime numbers of $P_m + P_n$ for various $m$ and $n$. They showed the following theorem.

\begin{theorem}[\cite{asplund}, Theorems 17, 18, and 19] \label{tight-join-asplund}
For $m \geq 4$ even and $n=2$, or $m \geq n$ and $n=3$ or $4$, the minimum coprime number of $P_m+P_n$ is given by
\begin{equation} \label{path-eq}
\mathfrak{pr}(P_m+P_n) = \begin{cases} m+2n-2 & \text{if } m \text{ is odd}\\ m+2n-1 & \text{if } m \text{ is even}.\end{cases}
\end{equation}
\end{theorem}
They proceeded to show that \eqref{path-eq} holds for $2 \leq n \leq 10$ if $m>118$.
They conjecture that this result extends to all $n$ as long as $m$ is sufficiently large.

\begin{conjecture}[\cite{asplund}, Conjecture 2] \label{conj3}
For any positive integer $N$, there exists a positive integer $M$ such that for all $m>M$ and $2 \leq n \leq N$, the minimum coprime number of $P_m+P_n$ is given by
\begin{equation*}
\mathfrak{pr}(P_m+P_n) = \begin{cases} m+2n-2 & \text{if } m \text{ is odd} \\ m+2n-1 & \text{if } m \text{ is even}.\end{cases}
\end{equation*}
\end{conjecture}

The $k^\text{th}$ \textit{Ramanujan prime} $R_k$ is the least integer for which $\pi(x) - \pi(\frac{x}{2}) \geq k$ holds for all $x \geq R_k$. Ramanujan primes were introduced in \cite{ramanujan} as a generalization of Bertrand's Postulate; they are published as sequence A104272 in the OEIS \cite{oeis}.
We resolve Conjecture \ref{conj3} affirmatively, showing the following.
\begin{theorem} \label{pathjoin}
For any positive integer $N$, if $M \geq R_{N-1}-2N + 1$, then for all $m \geq M$ and $2 \leq n \leq N$, the minimum coprime number of $P_m+P_n$ is
\begin{equation*}
\mathfrak{pr}(P_m+P_n) = \begin{cases} m+2n-2 & \text{if } m \text{ is odd} \\ m+2n-1 & \text{if } m \text{ is even}.\end{cases}
\end{equation*}
\end{theorem}

Our results on $P_m+P_n$ automatically yield upper bounds on the minimum coprime numbers of the complete bipartite graphs $K_{m,n}$, and our constructions also generalize to the join of two cycles and the join of a path and a cycle, enabling us to compute $\mathfrak{pr}(C_m+C_n)$, $\mathfrak{pr}(C_m+P_n)$, and $\mathfrak{pr}(P_m+C_n)$ for sufficiently large $m$. Our results on these classes of graphs are given in Theorems \ref{complete-bipartite} and \ref{cycle-cycle}, and Corollaries \ref{cycle-path} and \ref{path-cycle}.

We conclude this paper by providing an asymptotic for the minimum coprime number of a random subgraph, a topic which, to the best of our knowledge, has not been previously studied. Given a graph $G$ and $p \in (0,1)$, we obtain a probability distribution $G_p$ called a \emph{random subgraph} by taking subgraphs of $G$ with each edge appearing independently with probability $p$. When $G = K_n$, this is called the \emph{Erd\H{o}s--R\'enyi random graph}, denoted $G(n,p)$. We compute the following asymptotics for the minimum coprime number of the Erd\H{o}s--R\'enyi random graph.
\begin{theorem} \label{random-er}
We have
\begin{equation*}
\mathfrak{pr}(G(n,p)) \sim n \log n
\end{equation*}
almost surely, i.e., with probability tending to $1$ as $n \to \infty$.
\end{theorem}

As a corollary, we observe that
\begin{corollary}
For all $p$, $G(n,p)$ is almost surely not prime.
\end{corollary}

\subsection{Outline}
In Section \ref{prelim}, we state several straightforward results on the minimum coprime number of an arbitrary graph which we will later use in our proofs. In Section \ref{corona}, we discuss the graphs $K_n \odot \overline{K}_m$, proving Theorem \ref{corona-thm}. In Section \ref{join}, we focus on the graphs $P_m+P_n$, proving Theorem \ref{pathjoin}. We also consider the complete bipartite graphs $K_{m,n}$ and the graphs $C_m + C_n$ and $C_m+P_n$.
In Section \ref{random}, we discuss the minimum coprime number of the Erd\H{o}s--R\'enyi random graph, proving Theorem \ref{random-er}. We conclude by posing a number of open questions in coprime graph labeling in Section \ref{open}.

\section{Preliminaries} \label{prelim}
In this section we state several propositions relating the minimum coprime number to other graph properties. Recall that if $G$ and $H$ are graphs with $V(G) = V(H)$ and $E(G) \subseteq E(H)$, we call $H$ a \emph{spanning supergraph} of $G$ and $G$ a \emph{spanning subgraph} of $H$. The \emph{independence number} $\a(G)$ of $G$ is the size of the largest set of vertices $S$ in $V(G)$ such that no two vertices in $S$ are adjacent to one another. The \emph{chromatic number} $\chi(G)$ of $G$ is the least integer $k$ for which there exists a map $f: V(G) \to \{1, \dots, k\}$ such that if $(u,v) \in E(G)$ then $f(u) \neq f(v)$.

\begin{proposition} \label{subgraph}
Let $G$ and $H$ be graphs such that $H$ is a spanning supergraph of $G$. Then $\mathfrak{pr}(G) \leq \mathfrak{pr}(H)$.
\end{proposition}

\begin{proof}
Let $f: V(H) \to \{1, \dots, \mathfrak{pr}(H)\}$ be a minimum coprime labeling of $H$. As $V(G) = V(H)$, $f$ induces a coprime labeling of $G$. Hence $\mathfrak{pr}(G) \leq \mathfrak{pr}(H)$.
\end{proof}

The following corollary of Proposition \ref{subgraph} has been noted in numerous places in the literature.

\begin{corollary}
Any spanning subgraph of a prime graph is prime. Any spanning supergraph of a nonprime graph is nonprime.
\end{corollary}

\begin{proposition} \label{independence}
For any graph $G$, $\mathfrak{pr}(G) \geq 2(\abs{V(G)} - \a(G)) - 1$.
\end{proposition}

\begin{proof}
Under any coprime labeling of $G$, the vertices with even labels must form an independent set. Hence at most $\a(G)$ even integers may be used to label $G$. The lower bound follows from noting that at least $\abs{V(G)} - \a(G)$ odd integers must be used as labels.
\end{proof}

The following corollary of Proposition \ref{independence} has been stated elsewhere in the literature, for instance in \cite{fu} and \cite{seoud2}.

\begin{corollary} \label{prime-indep}
If $G$ is a prime graph, then $\a(G) \geq \left\lfloor \frac{\abs{V(G)}}{2} \right\rfloor$. 
\end{corollary}

\section{Minimum coprime numbers of $K_n \odot \overline{K}_m$} \label{corona}
In this section we prove Theorem \ref{corona-thm}.

\begin{proof}[Proof of Theorem \ref{corona-thm}]
Set $N = \max(p_{n-1},mn+n)$. It is apparent that $\mathfrak{pr}(K_n \odot \overline{K}_m) \geq mn+n$ as $\abs{V(K_n \odot \overline{K}_m)} = mn+n$. By Proposition \ref{subgraph}, we have $\mathfrak{pr}(K_n \odot \overline{K}_m) \geq \mathfrak{pr}(K_n)$. It is well known that $\mathfrak{pr}(K_n) = p_{n-1}$ (see \cite{asplund} for example). Therefore $\mathfrak{pr}(K_n \odot \overline{K}_m) \geq N$.

It now suffices to show that it is possible to construct a coprime labeling of $K_n \odot \overline{K}_m$ using only labels up to $N$. Denote $u_1, \dots, u_n$ to be the vertices of $K_n$ and $v_{i,1}, \dots, v_{i,m}$ to be the vertices of the $i^\text{th}$ copy of $\overline{K}_m$. Label $u_1$ with $1$ and the remaining $u_i$ with $p_{i-1}$, so that we have used the first $n-1$ primes as labels. We will label $v_{i,1}, \dots, v_{i,m}$ with integers not exceeding $mn+n$, excluding $1, p_1, \dots, p_{n-1}$.

After labeling the vertices of $K_n$ as indicated, there are at least $mn$ integers less than or equal to $N$ which have not yet been used as labels. Denote by $L_1$ the list of all such integers. We label the vertices $v_{i+1,k}$ for $1 \leq i \leq n-1$ as follows: for each $i$, select the first $m$ integers in $L_i$ which are coprime to $p_i$ (the label of $u_{i+1}$). Use these integers to label $v_{i+1,1}, \dots, v_{i+1,m}$, and remove them from $L_i$ to form a new list $L_{i+1}$. We may use any integers in $L_n$ to label $v_{1,1}, \dots, v_{1,m}$, as they are all coprime to $1$, the label of $u_1$.

We claim that at each step of this process there exist $m$ integers in $L_i$ which are coprime to $p_i$, so that this process in fact yields a coprime labeling for $K_n \odot \overline{K}_m$ using only integers up to $N$. Observe that $L_i$ contains at least $\abs{L_1}-m(i-1) = m(n-i+1)$ elements. Suppose for the sake of contradiction that there are at most $m-1$ elements in $L_i$ that are coprime to $p_i$; then the remaining $m(n-i)+1$ elements must all be multiples of $p_i$.

We consider the following cases. Note that by \cite{tout, youssef, asplund}, the statement in the theorem is known for all $n \leq 3$ or $m \leq 2$, so we may assume that $n \geq 4$ and $m \geq 3$.
\begin{itemize}
\item If $p_i > n$, we have
\begin{equation*}
p_i - \frac{ip_i}{n}+\frac{p_i}{mn} \geq p_i-\frac{(n-1)p_i}{n} +\frac{p_i}{mn} = \frac{p_i}{n} + \frac{p_i}{mn} > 1 + \frac{1}{m}.
\end{equation*}
\item If $p_i \leq n$ and $i \geq 3$, we have
\begin{equation*}
p_i - \frac{ip_i}{n}+\frac{p_i}{mn} \geq p_i-i + \frac{p_i}{mn} > 2 > 1 + \frac{1}{m},
\end{equation*}
as $p_i - i \geq 2$ for all $i \geq 3$.
\item If $p_i \leq n$ and $i=2$, we have
\begin{equation*}
p_i-\frac{ip_i}{n}+\frac{p_i}{mn} = 3 - \frac{6}{n} + \frac{3}{mn} > \frac{3}{2} > 1 + \frac{1}{m}.
\end{equation*}
\end{itemize}
Thus, in each case, the largest integer in $L_i$ is at least
\begin{equation*}
p_i(mn-mi+1) \geq mn\left(p_i-\frac{ip_i}{n}+\frac{p_i}{mn}\right) > mn\left(1+\frac{1}{m}\right) = mn+n,
\end{equation*}
which is a contradiction.
\end{proof}

\section{Minimum coprime numbers of $P_m+P_n$} \label{join}
In this section we prove Theorem \ref{pathjoin} and discuss the minimum coprime numbers of a number of graphs related to $P_m+P_n$. We will use the following theorem and lemma.

\begin{theorem}[\cite{montgomery}, Brun--Titchmarsh theorem] \label{bt}
Denote by $\pi(x; k, a)$ the number of primes at most $x$ that are equivalent to $a$ modulo $k$. Then
\begin{equation*}
\pi(x+y; k, a) - \pi(x; k, a) \leq \frac{2y}{\varphi(k) \log(y/k)}
\end{equation*}
whenever $y \geq k$, where $\varphi$ is the Euler totient function.
\end{theorem}

\begin{lemma} \label{eleven}
For all $x \geq 1$, there is a prime $p \in (x, 2x]$ such that $p \not\equiv 1, 10 \pmod{11}$.
\end{lemma}

\begin{proof}
By Corollary 3 to Theorem 2 of \cite{rosser}, we have
\begin{equation*}
\pi(2x) - \pi(x) > \frac{3}{5} \frac{x}{\log x}
\end{equation*}
whenever $x \geq 20.5$.
Combining this with Theorem \ref{bt} shows that the total number of primes in $(x,2x]$ which are not congruent to $1$ or $10$ modulo $11$ is at least
\begin{equation*}
\frac{3x}{5\log x} - \frac{2x}{5(\log x - \log 11)}.
\end{equation*}
When $x > 1331$, the above expression is positive, and we have the desired result. For $x \leq 1331$, we may manually verify the statement in the lemma.
\end{proof}

\begin{proof}[Proof of Theorem \ref{pathjoin}]
Set $L = 2 \left\lceil \frac{m-1}{2} \right\rceil +2n-1$. The lower bound $\mathfrak{pr}(P_m+P_n) \geq L$ follows from Proposition \ref{independence}, observing that $\a(P_m+P_n) = \left\lceil \frac{m}{2} \right\rceil$. Thus, to show that $\mathfrak{pr}(P_m+P_n) = L$, it suffices to construct a coprime labeling using only labels up to $L$.

By Theorem \ref{tight-join-asplund}, we know that such a labeling exists for $n \leq 4$.
Assume that $n \geq 5$ and hence $L \geq R_4 = 29$. Since $L \geq R_{n-1}$, we may label the vertices of $P_n$ with $1$ and $n-1$ primes between $\left\lceil \frac{L}{2} \right\rceil$ and $L$, which we denote $q_1, \dots, q_{n-1}$ in increasing order. By Lemma \ref{eleven}, we may assume that $q_\ell \not\equiv 1, 10\pmod{11}$ for some $1 \leq \ell \leq n-1$. Moreover, by our assumptions on $m$ and $n$, each of these primes is at least $17$. As $1$, $q_1, \dots, q_{n-1}$ are each coprime to all of the other integers up to $L$, any coprime labeling of $P_m$ using the remaining integers will yield a minimum coprime labeling for $P_m+P_n$.

If $q_1 > m+1$ we are done, as we may simply label the vertices of $P_m$ with $2, \dots, m+1$ in order. Otherwise, let $S_1$ be the ordered list of integers up to $L$, excluding $\{1, q_1, \dots, q_{n-1}\}$. We will inductively construct sets $S_i$ so that $\abs{S_i} = L - n - i + 1$ and every element less than $q_i$ contained in $S_i$ is coprime to its neighbors in $S_i$ that are also less than $q_i$.
It is clear that $S_1$ satisfies the conditions above. For $1 \leq i \leq n-2$, we construct $S_{i+1}$ from $S_i$ as follows.
\begin{itemize}
\item  If $q_{i+1} > q_i+2$ and $q_i > q_{i-1} + 2$ (if $i>1$), then $S_i$ contains the sequence $q_i-2, q_i-1, q_i+1, q_i+2$, where $q_i-2$ and $q_i+2$ are odd and composite. Observe that $3$ divides at most one of $q_i-2$ and $q_i+2$. If $3 \nmid q_i+2$, we can set $S_{i+1} = S_i \setminus \{q_i+1\}$, as $\gcd(q_i-2, q_i-1) = \gcd(q_i-1, q_i+2) = 1$. Otherwise, $3 \nmid q_i-2$, so we can set $S_{i+1} = S_i \setminus \{q_i-1\}$, as $\gcd(q_i-2, q_i+1) = \gcd(q_i+1, q_i+2) = 1$.

\item If $q_{i+1} = q_i+2$, then it suffices to set $S_{i+1} = S_i \setminus \{q_i+1\}$.

\item If $q_{i+1} > q_i+2$, $i>1$, and $q_i = q_{i-1}+2$, then $S_i$ contains the sequence $q_i-4, q_i-3, q_i+1, q_i+2$, where $q_i-4$ and $q_i+2$ are odd and composite. Observe that $5$ divides at most one of $q_i-4$ and $q_i+2$. If $5 \nmid q_i+2$, we can set $S_{i+1} = S_i \setminus \{q_i+1\}$, as $\gcd(q_i-4, p_i-3) = \gcd(q_i-3, q_i+2) = 1$ necessarily. Otherwise, $5 \nmid q_i-4$, so we can set $S_{i+1} = S_i \setminus \{q_i-3\}$, as $\gcd(q_i-4, q_i+1) = \gcd(q_i+1, q_i+2) = 1$.
\end{itemize}
We now construct a final ordered list $S_n \subseteq S_{n-1}$ such that every element in $S_n$ is coprime to its neighbors. If $q_{n-1} = L$ we are done, as we may set $S_n = S_{n-1}$. Otherwise, since $L$ is always odd, we have $q_{n-1} + 2 \leq L$, so $S_{n-1}$ contains the sequence $q_{n-1}-k-1$, $q_{n-1}-k$, $q_{n-1} + 1$, $q_{n-1}+2$, where $k = 1$ or $3$ depending on whether $q_{n-1}-2$ is composite or prime respectively. As in the cases above, it is always possible to obtain $S_n$ by removing one of $q_{n-1}-k$ and $q_{n-1}+1$.

If $q_{n-1}=L$, then we have $\abs{S_n} = \abs{S_{n-1}} = L-2n+2 \geq m$, and we obtain a minimum coprime labeling for $P_m+P_n$ by labeling the vertices of $P_m$ with the elements of $S_n$ in order. Otherwise, we have $\abs{S_n} = L-2n+1$. If $m$ is even, $L-2n+1 = m$ and we again obtain a minimum coprime labeling for $P_m+P_n$. If $m$ is odd, $L-2n+1 = m-1$, and we are short of one label. We resolve this by recalling that there exists some $q_\ell \not\equiv 1, 10 \pmod{11}$. To construct $S_n$, we have deleted one of $q_i-1$, $q_i+1$ for each $q_i$; hence there is some $x = q_\ell \pm 1$, $11 \nmid x$ that does not appear in $S_n$. As $q_1 > 23$, we may label the vertices of $P_m$ with the sequence
\begin{equation*}
x, 11, 12, 5, 4, 3, 8, 7, 6, 13, 10, 9, 14,
\end{equation*}
followed by all of the elements $x \in S_n$ such that $15 \leq x \leq L$, followed by $2$. This yields a minimum coprime labeling for $P_m+P_n$.
\end{proof}

The condition $M \geq R_{N-1} - 2N + 1$ in Theorem \ref{pathjoin} is sufficient but certainly not necessary; by Theorem \ref{tight-join-asplund}, it is known, for instance, that if $N=3$ or $4$, it suffices to set $M \geq N$.
The following theorem extends this result to $N=5$.

\begin{theorem} \label{path-5}
For $m \geq 5$, the minimum coprime number of $P_m + P_5$ is
\begin{equation*}
\mathfrak{pr}(P_m+P_5) = \begin{cases} m+8 & \text{if } $m$ \text{ is odd} \\ m+9 & \text{if } $m$ \text{ is even}.\end{cases}
\end{equation*}
\end{theorem}

\begin{proof}
We prove the theorem by casework.
\begin{itemize}
\item If $m \geq 20$, we are done by Theorem \ref{pathjoin} as $R_4 = 29$.
\item If $m=5$, we may label the vertices of the first path with the sequence $1, 3, 5, 9, 13$, and the vertices of the second path with $2, 7, 4, 11, 8$.
\item If $6 \leq m \leq 10$, we may label the vertices of $P_5$ with the sequence $3, 5, 9, 1, 15$, and the vertices of $P_m$ with the first $m$ integers in the sequence
\begin{equation*}
2, 7, 4, 11, 8, 13, 14, 17, 16, 19.
\end{equation*}
\item If $11 \leq m \leq 17$, we may label the vertices of $P_5$ with the sequence $1, 11, 13, 17, 19$, and the vertices of $P_m$ with the first $m$ integers in the sequence
\begin{equation*}
2, 3, 4, 5, 6, 7, 8, 9, 14, 15, 16, 21, 10, 23, 12, 25, 24.
\end{equation*}
\item If $m=18$ or $19$, we may label the vertices of $P_5$ with the sequence $1, 13, 17, 19, 23$, and the vertices of $P_m$ with the first $m$ integers in the sequence
\begin{equation*}
2, 3, 4, 5, 6, 7, 26, 9, 10, 11, 14, 15, 16, 21, 22, 25, 8, 27, 20.
\end{equation*}
\end{itemize}
\end{proof}

Our work on the join of two paths also offers insight on the complete bipartite graphs. Fu and Huang \cite{fu} proved that, for $m \leq n$, $K_{m,n}$ is prime if and only if $m \leq \pi(m+n) - \pi(\frac{m+n}{2}) + 1$. Seoud, Diab, and Elsakhawi \cite{seoud} showed that $K_{2,n}$ is prime for all $n$ and that $K_{3,n}$ is prime unless $n=3,7$. Berliner et al. \cite{berliner} provide all values of $n$ for $m \leq 13$ for which $K_{m,n}$ is prime and note that $K_{m,n}$ is prime for all $n \geq R_{m-1}-m$, as implied by \cite{fu}. They also ask about the behavior of $\mathfrak{pr}(K_{m,n})$ when $n < R_{m-1}-m$. By Proposition \ref{subgraph}, Theorem \ref{pathjoin} immediately implies the following bound on $\mathfrak{pr}(K_{m,n})$, which provides a partial answer to this question.

\begin{corollary}
If $n \geq R_{m-1}-2m+1$, then
\begin{equation*}
\mathfrak{pr}(K_{m,n}) \leq 2 \left\lceil \frac{n-1}{2} \right\rceil + 2m - 1.
\end{equation*}
\end{corollary}

In fact, a slight modification of the first part of the proof of Theorem \ref{pathjoin} enables us to show that, more generally,
\begin{theorem} \label{complete-bipartite}
If $m \leq n \leq R_{m-1}-m$, then $\mathfrak{pr}(K_{m,n}) \leq R_{m-1}$.
\end{theorem}

\begin{proof}
We label the vertices of $\overline{K}_m$ with $1$ and the $m-1$ largest primes up to $R_{m-1}$, which are, in particular, at least $\lceil \frac{R_{m-1}}{2} \rceil$. Hence they are each coprime to all of the other integers up to $R_{m-1}$. We may use any $n$ of the remaining integers to label the vertices of $\overline{K}_n$.
\end{proof}

However, we note that for $N \geq 6$, it is in general not sufficient to take $M \geq N$.

\begin{remark} \label{failure}
There exist positive integers $m \geq n$ such that
\begin{equation*}
\mathfrak{pr}(P_m+P_n) > 2\left\lceil \frac{m-1}{2} \right\rceil + 2n - 1.
\end{equation*}
\end{remark}

We now present a number of examples and counterexamples for Remark \ref{failure} that illustrate the complexity of the behavior of $\mathfrak{pr}(P_m+P_n)$ for $n \leq m \leq R_{n-1} - 2n$.

\begin{example}[$n=6$] \label{path6}
\begin{itemize}
\item For $6 \leq m \leq 9$,
\begin{equation*}
\mathfrak{pr}(P_m+P_6) = \begin{cases} m+10 & \text{if } m \text{ is odd} \\ m+11 & \text{if } m \text{ is even}.\end{cases}
\end{equation*}
This follows from labeling the vertices of $P_6$ with the sequence $3, 5, 9, 1, 15, 11$ and the vertices of $P_m$ with the first $m$ integers in the sequence
\begin{equation*}
2, 7, 4, 17, 8, 13, 14, 19, 16.
\end{equation*}

\item For $m=10$ or $11$, this equality does not hold. We prove this by contradiction: if $\mathfrak{pr}(P_{10}+P_6) = \mathfrak{pr}(P_{11}+P_6) = 21$, then there exist minimum coprime labelings of $P_{10}+P_6$ and $P_{11}+P_6$ using $5$ or $6$ even integers respectively and each of the odd integers up to $21$. The vertices of $P_6$ must be labeled with the set of integers $S = \{3, 5, 7, 9, 15, 21\}$, as no subset of $S$ is pairwise coprime to its complement in $S$, so all of the integers in $S$ must be used as labels for the same path, there are too many odds in $S$ to be used as labels in $P_{10}$ or $P_{11}$. However, because $S$ contains a total of $6$ integers, $4$ of which are multiples of $3$, it is not possible to arrange the elements of $S$ in a sequence such that adjacent elements are coprime. Therefore $\mathfrak{pr}(P_m+P_6) > 21$ for $m=10$ or $11$. The same reasoning shows that the minimum coprime numbers of $P_{10} + P_6$ and $P_{11}+P_6$ exceeds $22$. In fact, we may construct labelings to show that $\mathfrak{pr}(P_{10}+P_6) = \mathfrak{pr}(P_{11}+P_6) = 23$.
\end{itemize}
\end{example}

\begin{example}[$n=7$] \label{path7}
\begin{itemize}
\item For $m=7$, we have $\mathfrak{pr}(P_7+P_7) = 19$. This follows from labeling the vertices of one copy of $P_7$ with the sequence $3, 5, 9, 1, 15, 11, 13$, and the vertices of the second copy with the sequence $2, 7, 4, 17, 8, 19, 16$.

\item For $m=8$ or $10$, we have $\mathfrak{pr}(P_m+P_7) = m+13$. This follows from labeling the vertices of $P_7$ with the sequence $3, 5, 9, 7, 15, 1, 21$, and the vertices of $P_m$ with the first $m$ integers in the sequence
\begin{equation*}
2, 11, 4, 13, 8, 17, 16, 19, 22, 23.
\end{equation*}

\item For $m=9$, the lower bound is not tight. As in Example \ref{path6}, we prove this by contradiction. If $\mathfrak{pr}(P_9+P_7) = 21$, then there would exist a minimum coprime labeling of $P_9+P_7$ using $5$ of the even integers and all of the odd integers up to $21$. In particular, all the integers in the set $S = \{3, 5, 7, 9, 15, 21\}$ must be used as labels for the vertices of $P_7$, by the same reasoning as before. However, there are fewer than $5$ even integers up to $21$ that are coprime to all of the elements in $S$. We can construct a coprime labeling to show that $\mathfrak{pr}(P_9+P_7)=22$. This case demonstrates that the minimum coprime number of the join of certain paths is constrained by the even integers rather than the odds; it is also interesting that the lower bound may be tight for both $m-1$ and $m+1$ but fail to be tight for $m$.
\end{itemize}
\end{example}

Our methods in Theorem \ref{pathjoin} also enable us to prove the following results on the joins of cycles and paths.

\begin{theorem} \label{cycle-cycle}
For any positive integer $N$, if $M \geq R_{N-1}-2N + 1$, then for all $m \geq M$ and $n \leq N$, the minimum coprime number of $C_m+C_n$ is
\begin{equation*}
\mathfrak{pr}(C_m+C_n) = \begin{cases} m+2n & \text{if } m \text{ is odd} \\ m+2n-1 & \text{if } m \text{ is even}.\end{cases}
\end{equation*}
\end{theorem}

\begin{proof}
Label the vertices of $P_m+P_n$ as in the proof of Theorem \ref{pathjoin}. Then the labels for $P_n$ are all either $1$ or prime, so in particular the endpoints of $P_n$ have coprime labels and we may join them to form $C_n$ without violating the coprime labeling condition. If $m$ is even, then one endpoint of $P_m$ is labeled with $2$ and the other is labeled with an odd integer, so we may also join the endpoints of $P_m$ to form $C_m$ without violating the coprime labeling condition, obtaining a coprime labeling of $C_m+C_n$. This labeling is a minimum coprime labeling because we have $\mathfrak{pr}(C_m+C_n) \geq \mathfrak{pr}(P_m+P_n)$ by Proposition \ref{subgraph}.

If $m$ is odd, $\mathfrak{pr}(C_m+C_n) \geq m+2n$ by Proposition \ref{independence}, as $\a(C_m+C_n) = \lfloor \frac{m}{2} \rfloor$. Recall that the ordered list $S_n$ contains either $m$ or $m-1$ elements. If $\abs{S_n} = m-1$, the last element of $S_n$ is $L = m+2n-2$. Labeling the vertices of $P_m$ with the sequence $S_n$ with $m+2n$ appended and joining the endpoints of $P_m$ to form $C_m$ therefore yields a minimum coprime labeling for $C_m+C_n$. If $\abs{S_n} = m$, then $p_{n-1} = L = m+2n-2$. We may assume that $n \geq 5$ as the other cases are known by \cite{asplund}; therefore $p_1 > 17$. We consider the following cases.
\begin{itemize}
\item If $m+2n-4$ is composite, the last two elements of $S_n$ are $m+2n-4$ and $m+2n-3$. Labeling the vertices of $P_m$ with the sequence $S_n$, replacing $m+2n-3$ by $m+2n$, and joining the endpoints of $P_m$ to form $C_m$ therefore yields a minimum coprime labeling for $C_m+C_n$, as we have $\gcd(m+2n-4, m+2n) = \gcd(m+2n, 2) = 1$.

\item If $m+2n-4$ is prime, the last two elements of $S_n$ are $m+2n-6$ and $m+2n-5$. If $3 \nmid m+2n$, we may replace $m+2n-5$ with $m+2n$ to obtain a minimum coprime labeling for $C_m+C_n$, as above. If $3 \mid m+2n$, we join the endpoints of $P_m$ to form $C_m$, replace $m+2n-5$ with $m+2n$, and rearrange the labels surrounding the endpoints to read
\begin{equation*}
\dots, m+2n-6, 2, m+2n, 4, 3, 5, 6, \dots.
\end{equation*}
\end{itemize}
\end{proof}

The proof of Theorem \ref{cycle-cycle} immediately implies the following two corollaries on the join of a cycle and a path.
\begin{corollary} \label{cycle-path}
For any positive integer $N$, if $M \geq R_{N-1}-2N + 1$, then for all $m \geq M$ and $n \leq N$, the minimum coprime number of $C_m+P_n$ is
\begin{equation*}
\mathfrak{pr}(C_m+P_n) = \begin{cases} m+2n & \text{if } m \text{ is odd} \\ m+2n-1 & \text{if } m \text{ is even}.\end{cases}
\end{equation*}
\end{corollary}

\begin{corollary} \label{path-cycle}
For any positive integer $N$, if $M \geq R_{N-1}-2N + 1$, then for all $m \geq M$ and $n \leq N$, the minimum coprime number of $P_m+C_n$ is
\begin{equation*}
\mathfrak{pr}(P_m+C_n) = \begin{cases} m+2n-2 & \text{if } m \text{ is odd} \\ m+2n-1 & \text{if } m \text{ is even}.\end{cases}
\end{equation*}
\end{corollary}

\section{Minimum coprime number of a random subgraph} \label{random}
In this section, we prove Theorem \ref{random-er}.

\begin{lemma} \label{random-lem}
Denote $n:= \abs{V(G)}$. If $n$ is sufficiently large and $\a(G) < \sqrt{n}$, then $\mathfrak{pr}(G) \geq p_{n-\a(G)\sqrt{n}}$.
\end{lemma}

\begin{proof}
Under any coprime labeling of $G$, the vertices with labels which are multiples of some prime $p$ must form an independent set. Thus as we let $p$ range over $p_1, \dots, p_{\sqrt{n}}$ respectively, we obtain $\sqrt{n}$ (possibly non-disjoint) independent sets, each of which contains at most $\a(G)$ vertices. Hence there are at least $n-\a(G)\sqrt{n}$ vertices whose labels are not multiples of any prime up to $p_{\sqrt{n}}$. If any of the remaining labels exceeds $p_n$, then we are done. Otherwise, as the remaining labels only have prime factors larger than $p_{\sqrt{n}}$ and $p_{\sqrt{n}}^2 > p_n$ for large $n$, the remaining labels must all be prime. Using the next $n-\a(G)\sqrt{n}$ primes after $p_{\sqrt{n}}$ as labels yields the desired result.
\end{proof}

\begin{proof}[Proof of Theorem \ref{random-er}]
By a celebrated result of Bollob\'as and Erd\H{o}s \cite{bollobas}, we have $\omega(G(n,p)) \sim 2 \log_{1/p} n$ almost surely, where $\omega(G)$ denotes the \textit{clique number} of $G$, i.e., the size of the largest complete subgraph of $G$. As $\a(G(n,p)) = \omega(G(n,1-p))$, by the prime number theorem, $p_{n-\a(G)\sqrt{n}} \sim n \log n$, and using Lemma \ref{random-lem}, we have $\mathfrak{pr}(G(n,p)) \sim n \log n$ almost surely as $n \to \infty$.

\end{proof}

\section{Further directions} \label{open}
Here we pose a number of open questions in coprime graph labeling.

\begin{question}
We observed in Examples \ref{path6} and \ref{path7} that $\mathfrak{pr}(P_m+P_n)$ exhibits interesting behavior when $n \leq m \leq R_{n-1}-2n$. Is there a nice characterization of $\mathfrak{pr}(P_m+P_n)$ in these cases, and in particular, is it possible to predict when $\mathfrak{pr}(P_m+P_n) = 2 \left\lceil \frac{m-1}{2} \right\rceil +2n-1$?
\end{question}

\begin{question}
Is it sufficient to take $M \geq N$ in Theorem \ref{pathjoin} for sufficiently large $N$? This is conjectured in \cite{asplund}.
\end{question}

\begin{question}
Can we improve on the bounds for $\mathfrak{pr}(K_{m,n})$ in Theorem \ref{complete-bipartite}? In particular, can we obtain sharper bounds on $\mathfrak{pr}(K_{n,n})$? 
\end{question}

\begin{question}
The Cartesian product $P_m \square P_n$ is called a \emph{grid graph}. In particular, if $m=2$, the graph $P_2 \square P_n$ is called a \emph{ladder}. Dean \cite{dean} and Ghorbani and Kamali \cite{ghorbani} showed independently that all ladders are prime, resolving a conjecture of Varkey.
Other grid graphs have been shown to be prime, including $P_m \square P_n$ if $m \leq n$ and $n$ is prime \cite{sundaram}, and a few other cases in \cite{kanetkar}. It it true that $P_m \square P_n$ is prime for all $m$ and $n$? This would settle a conjecture in \cite{sundaram}.
\end{question}

\begin{question}
There has been a substantial amount of research conducted on the clique number and independence number of a random subgraph. Would any of these results enable us to obtain lower bounds on $\mathfrak{pr}(G_p)$ for arbitrary $G$?
\end{question}

\begin{question}
For arbitrary $G$, the trivial upper bound $\mathfrak{pr}(G_p) \leq \mathfrak{pr}(G)$ is asymptotically tight, as we may observe in the case where $G$ is prime. Is it possible to obtain a better upper bound on $\mathfrak{pr}(G_p)$ for specific classes of $G$?
\end{question}

\section{Acknowledgments}
This research was conducted at the University of Minnesota, Duluth REU and was supported by NSF/DMS grant 1659047 and NSA grant H98230-18-1-0010. The author would like to thank Joe Gallian for suggesting this problem, and Aaron Berger and Pat Devlin for numerous helpful conversations.

\bibliographystyle{abbrv}
\bibliography{citations}

\end{document}